\documentclass[a4paper,11pt]{amsart}

\usepackage[usenames,dvipsnames]{pstricks}
\usepackage{epsfig}
\usepackage{pst-grad} 
\usepackage{pst-plot} 
\usepackage{amssymb}
\usepackage{mhequ}
\usepackage{cite}

\DeclareFontFamily{OT1}{rsfs}{}
\DeclareFontShape{OT1}{rsfs}{m}{n}{ <-7> rsfs5 <7-10> rsfs7 <10->
rsfs10}{} \DeclareMathAlphabet{\mathscr}{OT1}{rsfs}{m}{n}

\newcommand{\zHkpp}{\zHk_{\phi,\psi}}

\newcommand{\zHk}{\zH^k}

\newcommand{\eq}[1]{\eqref{#1}}

\newcommand{\bel}[1]{\begin{equation}\label{#1}}
\newcommand{\beal}[1]{\begin{eqnarray}\label{#1}}
\newcommand{\beadl}[1]{\begin{deqarr}\label{#1}}
\newcommand{\eeadl}[1]{\arrlabel{#1}\end{deqarr}}
\newcommand{\eeal}[1]{\label{#1}\end{eqnarray}}
\newcommand{\eead}[1]{\end{deqarr}}
\newcommand{\eea}{\end{eqnarray}}
\newcommand{\eeaa}{\end{eqnarray*}}
\newcommand{\beaa}{\begin{eqnarray*}}
\newcommand{\be}{\begin{equation}}
\newcommand{\ee}{\end{equation}}

\DeclareFontFamily{OT1}{rsfs}{}
\DeclareFontShape{OT1}{rsfs}{m}{n}{ <-7> rsfs5 <7-10> rsfs7 <10->
rsfs10}{} \DeclareMathAlphabet{\mycal}{OT1}{rsfs}{m}{n}

\newcommand{\mcO}{\mathscr O}
\newcommand{\mcU}{\mathscr U}

\newcounter{mnotecount}[section]

\newcommand{\N}{{\Bbb N}}

\newcommand{\rmnote}[1]{}

\newcommand{\Ric}{\operatorname{Ric}}

\newcommand{\zH}{\mathring H}
\newcommand{\Lpsig}{L^2_{\psi}(g)}

\newcommand{\Lpsi}{L^2_{\psi}}

\newcommand{\Lpsim}{\zH^m_{\phi,\psi}}

\newcommand{\maclKzo}{{\mathcal K}^{\bot}}
\newcommand{\Hpsi}[1]{\zH^{#1}_{\phi,\psi}}

\def\mysavedown#1{\edef\mysubs{\mysubs#1}}
\def\mysaveup#1{\edef\mysups{\mysups#1}}
\def\mydown#1{{\mytensor}_{\vphantom{\mysubs}#1}}
\def\myup#1{{\mytensor}^{\vphantom{\mysups}#1}}
\def\tensor#1#2{
  #1
  \def\mytensor{\vphantom{#1}}
  \def\mysubs{\relax}
  \def\mysups{\relax}
  \let\down=\mysavedown
  \let\up=\mysaveup
  #2
  \let\down=\mydown
  \let\up=\myup
  #2
  }

\newcommand{\Tr}{\operatorname{Tr}}

\newcommand{\R}{\mathbb R}

\renewcommand{\setminus}{\smallsetminus}
\renewcommand{\emptyset}{\varnothing}

\renewcommand{\div}{\operatorname{div}}
\renewcommand{\epsilon}{\varepsilon}
\renewcommand{\hat}{\widehat}

\def\crn#1#2{{\vcenter{\vbox{
        \hbox{\kern#2pt \vrule width.#2pt height#1pt
           }
          \hrule height.#2pt}}}}

\renewcommand{\H}{\mathbb H}

\renewcommand{\hbar}{{\overline h}}

\newcommand{\pre}[2]{{{\vphantom{#2}}^{#1}}\kern-.2ex{#2}}

\theoremstyle{plain}
\newtheorem{theorem}{\sc Theorem}[section]
\newtheorem{lemma}[theorem] {\sc Lemma}
\newtheorem{proposition}[theorem]{\sc Proposition}
\newtheorem{corollary}[theorem] {\sc Corollary}

\theoremstyle{definition}
\newtheorem{definition}[theorem]{Definition}

\newtheorem{remark}[theorem]{\sc  Remark\rm}

\numberwithin{equation}{section}

\date{May 17, 2011}

\begin{document}

\title[Compactly supported solutions of   underdetermined  elliptic PDE] {Smooth compactly supported solutions of  some underdetermined  elliptic PDE, with gluing applications }
\author[E.
Delay]{Erwann Delay} \address{Erwann Delay, Laboratoire d'analyse
non lin\'eaire et g\'eom\'etrie, Facult\'e des Sciences, 33 rue
Louis Pasteur, F84000 Avignon, France}
\email{Erwann.Delay@univ-avignon.fr}
\urladdr{http://www.math.univ-avignon.fr/Delay}
\begin{abstract}
We give sufficient conditions for some underdetermined elliptic PDE of any order to construct smooth compactly supported solutions.
In particular we show that two smooth elements in the kernel of certain  underdetermined linear elliptic operators  $P$ can be glued in a chosen region in order
to obtain a new smooth solution. This new solution is exactly equal to  the {  initial} elements outside the gluing region.
This result completely contrast{s} with the usual unique continuation for determined or overdetermined elliptic operators. As a corollary we obtain
compactly supported solutions in the kernel of $P$ and also solutions vanishing in a chosen relatively compact open region.
We apply the result for natural geometric and physics context{s} such as divergence free field{s} or TT-tensors.
\end{abstract}

\maketitle

\noindent
{\bf Keywords :} Undetermined elliptic PDE,  compactly supported solutions, gluing.\\

\noindent
{\bf MSC 2010 :} 35J99, 58J99, 35Q35, 35Q60, 35Q75

\tableofcontents
\section{Introduction}\label{section:intro}
Determined and overdetermined elliptic operators are particularly studied and a lot of very nice results
are known. This is in part due  to {the} rigidity of solutions. A classical result about those
operator{s} is the unique continuation property (see \cite{NUW} for a recent result).
{At the opposite}, {  much less is known about} underdetermined elliptic operators, ie. with {\it surjective but not injective principal symbol}
(see however \cite{BEM} for instance).
In the present paper we are first interested about the following natural PDE problem:
let $P$ be an underdetermined elliptic operators with smooth coefficients and let $f$ be a smooth
compactly supported source. One wants to construct a smooth compactly supported solution $U$ to
$$
PU=f.
$$

We work on a smooth (ie. $C^\infty$) Riemannian manifold $(M,g)$. We do not assume that $(M,g)$ is connected {  neither} complete nor compact.
Let $\Omega$ be a relatively  compact open set with smooth boundary.
Let $P$ be an underdetermined elliptic operator  with smooth coefficients, of order $m$, acting on natural tensors fiber bundles over $M$.
Let $P^*$ be its  formal $L^2$ adjoint.
Before {  giving} the theorem we need to give some definitions, which are {required for the statement of} the hypotheses.
\begin{definition}
We {say} that $P^*$ satisfies the {Kernel Restriction Condition} (KRC) if any element in the kernel $\mathcal K$ of $P^*$
on $\Omega$ is in  $C^{m-1}(\overline\Omega)$.
\end{definition}
We refer the reader to section \ref{SwSs} for the definition of the weighted Sobolev spaces.
\begin{definition}\label{Poincare}
We will {say} that $P^*$ satisfies the {Asymptotic Poincar\'e Inequality} (API) if there exists a compact set $K\subset \Omega$ and a constant $C$, such that for
all  $C^\infty$ sections $u$, supported  in $\Omega\setminus K$ we have
\be
\label{it1ker}
C\|\phi^mP^*(u)\|_{\Lpsi} \geq
\|u\|_{\Lpsi}\;.
\ee
\end{definition}

Our theorem is now stated as follows:
\begin{theorem}\label{theorpufi}
Assume that $P^*$ satisf{ies}
(API)
 and
(KRC).
Let $f\in C^\infty(M)$ be a smooth source with compact support in $\Omega$, such that
$$
\int_\Omega \langle f,v\rangle=0,
$$
for any $v$ in the kernel of $P^*$ on $\Omega$. Then there exist $U\in C^\infty(M)$ with compact
support in  $\overline \Omega$ such that $PU=f$.

\end{theorem}

The assumption $\int_\Omega \langle f,v\rangle=0$ is obviously necessary in order to have the announced conclusion.

The basic example of an operator satisfying the conditions of the above theorem is a linear operator of order $m>0$, with smooth coefficients, and  of the form
$$
P^*=A(\nabla^{(m)}u)+\mbox{lot},
$$
where $A$ is an injective linear operator with smooth coefficients
(see section \ref{secex}).

Let us briefly give the idea of the proof.
We first work on $\Omega$. {Using weighted spaces, we show that we can solve
$$
PU=f
$$ with some $U=\zeta P^*u$ up to
a (weighted) projection on the orthogonal to the kernel of $P^*$}.
In this construction, $u$ and its derivatives might blow up at the boundary
but {  the smooth positive function $\zeta$ and its derivatives {vanish} more}, {so that} $U$ and its
derivatives vanish at the boundary.
In a second step, using  integrations by part, for any sections $U$ and $v$ in their respective bundles,
$$
\int_\Omega \langle PU,v\rangle=\int_\Omega \langle U,P^*v\rangle+\int_{\partial\Omega}B(U,v),
$$
where $B$ is a bilinear operator of order $m-1$.
We then show {that} the projection onto the kernel vanishes.
{Here} (KRC)  leads to the vanishing of the boundary terms.

We apply our theorem to show how elements in the kernel of $P$ are flexible  in the sense that
we can, in a chosen region, glue two smooth solution{s}  in order to construct a third one.
In particular taking one of the two solution{s} to be zero, one can truncate a solution to obtain
a solution either of compact support or vanishing on a chosen compact  set.
The operators studied here are linear, but the technique{s} can certainly be adapted
to {certain} non linear context{s} as it has already been done in the special case of the constraint map (see eg. \cite{Corvino},
\cite{CorvSchoen}, \cite{CD2}, \cite{ChrCorvIsen}).

Let $\Omega_i$, $i=1,2$ be open subsets of $M$   such that
$\Omega:=\Omega_1\cap\Omega_2\neq\emptyset $, $\overline\Omega$ is compact and the boundary
of $\Omega$ is a smooth submanifold decomposed in two disjoint smooth submanifolds : $\partial\Omega=\partial_1\Omega\cup\partial_2\Omega$ with
$\partial_i\Omega\subset\Omega_i$, $i=1,2$ (see figure A bellow).\\

\vspace{1cm}
%

\scalebox{0.5} 
{
\begin{pspicture}(0,-4.4025)(19.48,4.4225)
\psbezier[linewidth=0.04](0.0,3.8575)(0.0,3.0575)(19.46,1.4175)(19.46,2.2175)
\psbezier[linewidth=0.04](0.1,-3.4625)(0.1,-4.2625)(11.305545,-1.2476593)(12.3,-1.1425)(13.294456,-1.0373406)(19.244516,0.014647029)(19.32,-0.9825)
\psbezier[linewidth=0.04](0.12,1.6175)(0.12,0.8175)(7.936487,2.012128)(8.04,1.0175)(8.143513,0.022871878)(1.0080229,-1.851651)(0.2,-1.2625)
\psellipse[linewidth=0.04,dimen=outer](4.38,2.1875)(0.2,0.79)
\psellipse[linewidth=0.04,dimen=outer](4.29,-1.8725)(0.19,1.17)
\psellipse[linewidth=0.04,dimen=outer](13.79,0.5875)(0.27,1.53)
\psdots[dotsize=0.12,linecolor=white](4.48,2.7775)
\psdots[dotsize=0.12,linecolor=white](4.52,2.5175)
\psdots[dotsize=0.12,linecolor=white](4.54,2.2375)
\psdots[dotsize=0.12,linecolor=white](4.52,1.9775)
\psdots[dotsize=0.12,linecolor=white](4.5,1.7575)
\psdots[dotsize=0.12,linecolor=white](4.48,1.5775)
\psdots[dotsize=0.12,linecolor=white](4.4,-0.9225)
\psdots[dotsize=0.12,linecolor=white](4.42,-1.2025)
\psdots[dotsize=0.12,linecolor=white](4.44,-1.4425)
\psdots[dotsize=0.12,linecolor=white](4.44,-1.7625)
\psdots[dotsize=0.12,linecolor=white](4.46,-2.0025)
\psdots[dotsize=0.12,linecolor=white](4.48,-2.2425)
\psdots[dotsize=0.12,linecolor=white](4.46,-2.4825)
\psdots[dotsize=0.12,linecolor=white](4.4,-2.7025)
\psdots[dotsize=0.12,linecolor=white](13.94,1.8975)
\psdots[dotsize=0.12,linecolor=white](13.96,1.6575)
\psdots[dotsize=0.12,linecolor=white](13.98,1.3975)
\psdots[dotsize=0.12,linecolor=white](14.0,1.1175)
\psdots[dotsize=0.12,linecolor=white](14.02,0.8375)
\psdots[dotsize=0.12,linecolor=white](14.02,0.5975)
\psdots[dotsize=0.12,linecolor=white](14.02,0.3175)
\psdots[dotsize=0.12,linecolor=white](14.02,0.0575)
\psdots[dotsize=0.12,linecolor=white](14.0,-0.2225)
\psdots[dotsize=0.12,linecolor=white](13.96,-0.4825)
\psdots[dotsize=0.12,linecolor=white](13.9,-0.7225)
\usefont{T1}{ptm}{m}{n}
\rput(9.665781,0.5775){\Large $\Omega$}
\usefont{T1}{ptm}{m}{n}
\rput(3.2448437,0.3175){\Large $\partial_1\Omega$}
\usefont{T1}{ptm}{m}{n}
\rput(15.404843,-2.0425){\Large $\partial_2\Omega$}
\psline[linewidth=0.04cm](19.4,3.5375)(4.4,3.5775)
\psline[linewidth=0.04cm](0.14,-4.2425)(13.78,-4.2625)
\psline[linewidth=0.04cm](4.44,3.5575)(4.7,3.4575)
\psline[linewidth=0.04cm](4.42,3.5975)(4.68,3.7175)
\psline[linewidth=0.04cm](13.58,-4.1225)(13.76,-4.2225)
\psline[linewidth=0.04cm](13.76,-4.2825)(13.62,-4.3825)
\usefont{T1}{ptm}{m}{n}
\rput(13.151563,4.1975){\Large $\Omega_2$}
\usefont{T1}{ptm}{m}{n}
\rput(8.692344,-3.3425){\Large $\Omega_1$}
\psline[linewidth=0.04cm](13.64,-0.0225)(15.14,-1.6025)
\psline[linewidth=0.04cm](13.62,-0.0025)(13.68,-0.2425)
\psline[linewidth=0.04cm](13.62,0.0175)(13.82,-0.0425)
\psline[linewidth=0.04cm](3.12,0.6975)(4.12,2.2175)
\psline[linewidth=0.04cm](4.12,2.2175)(4.08,2.0375)
\psline[linewidth=0.04cm](4.1,2.2175)(3.98,2.1375)
\psline[linewidth=0.04cm](3.14,-0.0025)(4.02,-1.4625)
\psline[linewidth=0.04cm](4.0,-1.2825)(4.04,-1.4225)
\psline[linewidth=0.04cm](4.04,-1.4225)(3.9,-1.3825)
\psline[linewidth=0.04cm](4.14,2.5775)(4.24,2.7175)
\psline[linewidth=0.04cm](4.24,2.7175)(4.32,2.5975)
\psline[linewidth=0.04cm](4.1,-1.8625)(4.02,-2.1025)
\psline[linewidth=0.04cm](4.1,-1.8625)(4.22,-2.0425)
\psline[linewidth=0.04cm](13.44,1.0375)(13.5,0.8575)
\psline[linewidth=0.04cm](13.54,0.8975)(13.64,1.0575)
\end{pspicture}}
\vspace{5mm}

\begin{center}{Figure A}
\end{center}
\vspace{1cm}


\begin{definition}\label{defflux}
Let $V\in C^{m}(\Omega_1)$ and $W\in C^{m}(\Omega_2)$. We say that the {Flux Condition} (FC) hold{s} if (KRC) holds and if
for all $v\in\mathcal K$,
$$\int_{\partial_2\Omega}B(W,v)=\int_{\partial_1\Omega}B(V,v).$$
\end{definition}

\begin{theorem}\label{maintheorem}
Assume that $P^*$ satisf{ies}
(API)
 and
(KRC).
Let $V\in C^\infty({\Omega_1})$ and $W\in C^\infty({ \Omega_2})$ be two smooth elements in the
 kernel of $P$. Let $\chi$ be a smooth cutoff function equal to $1$ near $\overline{\Omega_1\backslash \Omega}$ and equal to $0$ near  $\overline{\Omega_2\backslash \Omega}$. If  $V$ and  $W$ satisfy the flux condition (FC),
 then there exist{s} $U\in C^\infty({ \Omega_1\cup\Omega_2})$, supported in $\overline\Omega$ such that
 $\chi V+(1-\chi)W +U$ lies in the kernel of $P$.
\end{theorem}
More general gluing can be done, see remark \ref{collegene}.

There are many other interesting situation{s} where this theorem is useful: we included few examples of section \ref{examples} but it is certain that there are many other applications.
Note that if the flux of $V$ on $\partial_1\Omega$ is zero, i.e.
$$
\int_{\partial_1\Omega}B(V,v)=0\;,\;\;\forall v\in \mathcal K=\;\mbox{ker}\; P^*,
$$ then one can take $W=0$: this shows that it is always possible to truncate a solution or make it
vanish on a chosen region. This allows one to construct solutions on quotients or on connected sums.
This is also a powerful tool to prove the density of compactly supported solutions in the kernel and
 to prove that this last set of solutions is infinite dimensional.

Let us illustrate the applications with two examples.
If $P$ is the divergence operator acting on one forms, as it is natural for instance in fluid mechanic  or in eletromagnetism, we obtain the
\begin{corollary}
For $P=d^*$, the divergence operator acting on one forms, the conclusions of  Theorems \ref{theorpufi} and \ref{maintheorem}   hold.
\end{corollary}
As a consequence, we exhibit  in  section \ref{secdivfree} a procedure to glue some electric field to the electric field
surrounding a point  charge.\\

In  general relativity, the constraint equations are the initial data constraint for  the evolution to be Einstein \cite{Carges}.
When constructing CMC initial data, it is first natural to construct TT-tensors, i.e. trace free and divergence
free symmetric covariant two tensor{s}. In that context, we get
\begin{corollary}
Let $P=\mathcal D^*$ be the divergence operator acting on trace free symmetric covariant two tensors. If $n\geq 3$,
then the conclusions of  theorem \ref{theorpufi} and \ref{maintheorem} hold.
\end{corollary}

For $n=2$, the operator $\mathcal D^*$  is (determined) elliptic and the corollary above is not true, see section \ref{secTT}.

As in the case of electric field, in section \ref{secTT}, we  glue some TT-tensors with a Beig-Bowen-York tensor.
Note  that in all the results above, in order to construct TT-tensors one need{s} to assume there is no conformal Killing field ($\ker \mathcal D=\{0\}$),
which is not the case here.
On {a} flat torus for instance, one can  truncate
elements in the kernel on $\R^n$ as before and reproducing them by periodicity.
On some quotients of the hyperbolic space, one can also  truncate a TT-tensors of $\H^n$
making its  support lie   in a fundamental domain and again transport it by the action.
The same can be done on the Sphere or on quotients thereof or on other Riemannian manifolds
with symmetries.

\section{Weighted
spaces}\label{SwSs}

We will use the spaces already introduced in {the} appendix of \cite{CD2} in the special case
of {a} compact boundary. We keep the general notation of \cite{CD2} for eventual
adaptations of the paper in other context{s} such as asymptotically euclidian
or asymptotically hyperbolic {manifolds}.

Let $x\in C^{\infty}(\overline{\Omega})$ be a (non negative) defining function of the boundary  $\partial\Omega=x^{-1}(\{0\})$.

Let $a\in\N$, $s\in\R$, $s\neq0$ and let us define $$\phi=x^2\;, \quad \psi=x^{2(a-n/2)}e^{-s/x}\mbox{ and }  \varphi=x^{2a}e^{-s/x}.$$

For $k\in \N$ let $H^k_{\phi,\psi}$ be the space of $H^k_{{loc}}$
functions or tensor fields such that the norm
 \be \label{defHn}
 \|u\|_{H^k_{\phi,\psi}}:=
(\int_M(\sum_{i=0}^k \phi^{2i}|\nabla^{(i)}
u|^2_g)\psi^2d\mu_g)^{\frac{1}{2}} \ee is finite, where
$\nabla^{(i)}$ stands for the tensor $\underbrace{\nabla ...\nabla
}_{i \mbox{ \scriptsize times}}u$, with $\nabla$ --- the
Levi-Civita covariant derivative of $g$;  For
$k\in \N$ we denote by $\zH^k_{\phi,\psi}$ the closure in $H^k_{\phi,\psi}$ of the
space of $H^k$ functions or tensors which are compactly (up to a
negligible set) supported in $\Omega$, with the norm induced from
$H^k_{\phi,\psi}$.
The $\zH^k_{\phi,\psi}$'s are Hilbert spaces with the obvious scalar product
associated to the norm \eq{defHn}. We will also use the following
notation
$$
\quad \zH^k  :=\zH^k  _{1,1}\;,\quad
L^2_{\psi}:=\zH^0_{1,\psi}=H^0_{1,\psi}\;,
$$ so that $L^2\equiv \zH^0:=\zH^0_{1,1}$. We  set
$$
W^{k,\infty}_{\phi}:=\{u\in W^{k,\infty}_{{{loc}}} \mbox{ such that }
\phi^i|\nabla^{(i)}u|_g\in L^{\infty}, i=0,...,k\}\;,
$$
with the obvious norm.

For  $k\in\N$ and $\alpha\in [0,1]$, we define
$C^{k,\alpha}_{\phi,\varphi}$ the space of $C^{k,\alpha}$
functions or tensor fields  for which the norm
$$
\begin{array}{l}
\|u\|_{C^{k,\alpha}_{\phi,\varphi}}=\sup_{x\in
M}\sum_{i=0}^k\Big(
\|\varphi \phi^i \nabla^{(i)}u(x)\|_g\\
 \hspace{3cm}+\sup_{0\ne d_g(x,y)\le \phi(x)/2}\varphi(x) \phi^{i+\alpha}(x)\frac{\|
\nabla^{(i)}u(x)-\nabla^{(i)}u(y)\|_g}{d^\alpha_g(x,y)}\Big)
\end{array}$$ is finite.

\begin{remark}
In our setting, it is more usual to set $\phi=x$ and for $\psi$ and $\varphi$ a power of $x$. This can be done here also as long as we work with finite differentiability. We choose to take the exponential weight
to treat all the case{s} in the same way. Note also  for applications that condition (API) is in general {easier}
 to obtain with exponential weight.
\end{remark}

\section{Isomorphism properties}

We are interested in the surjectivity  of $P$ applied to sections $U$ that vanish exponentially at the
boundary. For the construction of a right-inverse of $P$, we will use the formal $L^2$ adjoint $P^*$.
Duality in $L^2$ makes natural to look at $P^*$ acting  on sections $u$ that can blow up exponentially
on the boundary.

\begin{proposition}\label{RI}
There exist{s} a constant $C$ such that for all $u\in\zH^{m}_{\phi,\psi}$,
\be\label{RIC}
C\left(\|\phi^mP^*u\|_{\Lpsi}+\|u\|_{\zH^{m-1}_{\phi,\psi}}\right) \geq
\|u\|_{\Lpsim}, \ee
\end{proposition}

\begin{proof}
 By a density argument,     we need only
 to prove \eq{RIC} for   smooth compactly supported
fields.
The proof proceeds in two steps: we first show that a similar inequality \eq{nwRIC} is valid in the usual (non weighted) Sobolev spaces
and then we show that \eq{nwRIC} implies the  estimate in  weighted spaces \eq{RIC}.

\begin{lemma}
 {  $P^*$ having  injective symbol and smooth coefficients up to the boundary}, there exist
 a constant $C'$ such that for all smooth  sections $u$ compactly supported in $\Omega$,
 \be\label{nwRIC}
C'\left(\|P^*(u)\|_{L^2}+\|u\|_{H^{m-1}}\right) \geq
\|u\|_{L^2}, \ee

\end{lemma}
\begin{proof}
This is exactly the   lemma 1.0.2. of \cite{Hormander1966} with $K$ there equal to $\overline \Omega$ here.

An alternative way to get this result is as follows:
 We smoothly prolong the set $\Omega$, together
with the operator $P$  (keeping the symbol surjective), to a  compact riemannian manifold $M$ without boundary. From  \cite{Hormander3} Theorem 19.5.1 (see also the remark below Theorem 19.5.2 there),
the operator $P^*: H^m(M)\rightarrow L^2(M)$ is  semi-Fredholm (ie. has finite dimensional kernel and closed range)
thus the estimate (see eg. equation (19.5.1) in \cite{Hormander3})
$$
\|u\|_{H^m}\leq C'(\|P^*u\|_{L^2}+\|u\|_{H^{m-1}}),
$$
is valid on $M$, so in particular for smooth sections $u$ compactly supported in $\Omega$.\\

Remark  that for a first order operator with injective symbol, one can use  \cite{Taylor1} proposition 12.1
(note that the boundary condition there is stated to obtain \eq{nwRIC} for all $u\in H^1$ and not only
for $u\in\mathring H^1$).
\end{proof}

\begin{lemma}
The estimate \eq{nwRIC} implies the estimate \eq{RIC}
\end{lemma}
\begin{proof}
We first remark that the    condition
$$
\|u\|_{H^{m-1}}+\|P^*u\|_{L^2}\geq c\|u\|_{H^m},
$$
is equivalent to
$$
\|u\|_{H^{m-1}}+\|P^*u\|_{L^2}\geq c'\|\nabla^{(m)}u\|_{L^2}.
$$
(A similar equivalence is also true in our weighted Sobolev spaces.)
We put $u=\phi^m\psi w$ in the last inequality.
We get
\begin{eqnarray*}
\|\nabla^{(m)}u\|_{L^2}&=&\|\phi^m\psi \nabla^{(m)} w+\sum_{i=0}^{m-1}C_i\nabla^{(m-i)}(\phi^m\psi)\nabla^{(i)}w\|_{L^2}\\
&\geq&\|\phi^m\nabla^{(m)} w\|_{L^2_\psi}-c_1\|w\|_{H^{m-1}_{\phi,\psi}}.
\end{eqnarray*}
Similarly, we find
\begin{eqnarray*}
\|P^*u\|_{L^2}&=&\|\phi^m\psi P^* w+lot\|_{L^2}\\
&\leq&\|\phi^m P^* w\|_{L^2_\psi}+c_2\|w\|_{H^{m-1}_{\phi,\psi}},
\end{eqnarray*}
because  the {  lower order terms} are of the form : smooth coefficient on $\overline \Omega$ times derivatives of $\phi^m\psi$ times derivatives
of $w$ and each of them is bounded by a term of the form $\|\phi^{i}\nabla^{(i)}w\|_{L^2_{\psi}}$ with $i\leq m-1$
(see appendix for details).
\end{proof}

The  two lemma above give the proposition \ref{RI}.
\end{proof}

\begin{proposition}\label{RI2}
There exist{s} a constant $C$ such that for all $u\in\zH^{m}_{\phi,\psi}$,
\be\label{RIC2}
C\left(\|\phi^mP^*(u)\|_{\Lpsi}+\|u\|_{L^2_\psi}\right) \geq
\|u\|_{\Lpsim}, \ee
\end{proposition}
\begin{proof}
By the interpolation inequality (recalled at the end of the appendix), for any positive constant $\epsilon$, there exist $C''>0$ such that
for all $u\in\zH^{m}_{\phi,\psi}$,
\be\label{InterpInequ}
\|u\|_{\zH^{m-1}_{\phi,\psi}}\leq \epsilon \|u\|_{\zH^{m}_{\phi,\psi}}+C''\|u\|_{L^2_\psi}.
 \ee
This inequality combined with \eq{RIC} proves the announced result.
\end{proof}

\begin{lemma}
 Let us assume that $P$ satisfies (API). Let $V$ be a relatively compact open subset of $\Omega$ containing $K$.   Then there exist{s} a constant $C'$ such that for all $u\in\zH^{m}_{\phi,\psi}$,
\bel{eL6}
C'\left(\|\phi^mP^*(u)\|_{\Lpsi}+\|u\|_{L^2(\overline V)}\right) \geq
\|u\|_{\Lpsim}.\ee
In particular the map $$\phi^m P^*:
\Lpsim\longrightarrow
{\Lpsi}
$$
has finite dimensional kernel.
\end{lemma}
\begin{proof}
This is now a classical argument, see \cite{Lee:fredholm} proof of lemma 4.10 for instance.
\end{proof}

%

\begin{proposition}\label{P:estproj}
Let
${\mathcal K   }$ be the kernel of
$$\phi^m P^*:
\Lpsim\longrightarrow
{\Lpsi},$$ and let ${\mathcal K
}^{\bot}$ be its $\Lpsi$-orthogonal.
If $P^*$ satisfies (API).
 then there exists
a constant $C"$ such that
for all $u\in {\mathcal K
}^{\bot} \cap
\Lpsim$ we have
\bel{broker}
C"\|\phi^mP^*(u)\|_{\Lpsi} \geq
\|u\|_{\Lpsim }\;.
\ee
\end{proposition}

\begin{proof} This is  a standard argument,
compare~\cite{choquet-bruhat:christodoulou:elliptic,AndElli,Lee:fredholm}:
assuming that the inequality fails,  there is a sequence
$(u_n)\in \zH^m_{\phi,\psi}\cap {\mathcal K  }^{\bot} $ with norm
$1$ such that $\|\phi^mP^*( u_n)\|_{\Lpsi}$ approaches
zero as $n$ tends to infinity. One obtains a contradiction with
injectivity on $\zH^m_{\phi,\psi}\cap {\mathcal K }^{\bot} $ by using
the Rellich-Kondrakov compactness on a relatively compact open
set, applying \eq{eL6}.
\end{proof}

\begin{remark}
As in \cite{Lee:fredholm} proof of lemma 4.1, on can also show that the map
$$\phi^m P^*:
\Lpsim\longrightarrow
{\Lpsi},$$
has closed range under (API). Also, by adapting the same proof, one shows that
 this map is semi-Fredholm iff (API) holds.
\end{remark}

Set
$$ {\mathcal L}_{\phi,\psi}:=
\psi^{-2} P\psi^2\phi^{2m} P^*
\;.$$

 We denote by $\pi_{\maclKzo }$ the $L^2_\psi$ projection
onto $\maclKzo $. We are now ready to prove:

\begin{theorem}\label{iso}
Let $k\ge 0$, and assume that (API)  {holds}. Then
 the map \bel{isoeq} \pi_{\maclKzo } {\mathcal L}_{\phi,\psi}
:{\maclKzo }\cap  \Hpsi{k+2m}
\longrightarrow {\maclKzo }\cap
\Hpsi{k} \ee is an  isomorphism.
\end{theorem}

 \begin{proof}
For $f\in {\maclKzo }\cap{\Lpsi}$,
let $\mathcal F$ be the following continuous functional defined
on ${\maclKzo }\cap \Lpsim$:
$$
{\mathcal F}(u):=\int_M (\frac{1}{2}|\phi^mP^*( u)|^2_g -\langle
u,f\rangle _g)\psi^2 d\mu_g\;;
$$
we set
$$
\mu_F=\inf_{u\in {\maclKzo }\cap \Lpsim}{\mathcal F}(u)\;.
$$
We claim that ${\mathcal F}$ is coercive: indeed,
Proposition~\ref{P:estproj} and the Schwarz inequality give
$$
\begin{array}{lll}
{\mathcal F}(u)&\geq& C(\|u\|_{\Lpsim })^2-
\|u\|_{\Lpsi}\|f\|_{\Lpsi}\\
&\geq & C(\|u\|_{\Lpsim })^2-
\|u\|_{\Lpsim }\|f\|_{\Lpsig}\;.
\end{array}
$$
Standard results on convex, proper, coercive, l.s.c.\ ({\em cf.,
e.g.,}~\cite[Proposition~1.2, p.~35]{EkelandTemam}) functionals
 show that $\mu_F$ is
achieved by some $u\in {\maclKzo }\cap \zH^m_{\phi,\psi}$ satisfying (the equality being trivial
for $w\in \mathcal K$)
\bel{p16}
\forall  w\in
 \zH^m_{\phi,\psi},
\int_M \left(\langle \phi^mP^*u,\phi^mP^*w\rangle
_g -\langle f,w\rangle
_g\right)\psi^2 d\mu_g=0.
\ee
It follows that  $u\in {\maclKzo }\cap \zH^m_{\phi,\psi}$ is a weak
solution of the equation
$$\psi^{-2} P\psi^2 \phi^{2m}P^*u=f.$$
The variational equation \eq{p16} satisfies the hypotheses of
\cite[Section~6.4, pp.~242-243]{Morrey}. By
elliptic regularity \cite[Theorem~6.4.3, p.~246]{Morrey} and by
standard scaling arguments ({\em cf.\/} appendix B of \cite{CD2}) for $f\in \zH^{k}_{\phi,\psi}$, we have $u\in \zH^{k+2m}_{\phi,\psi}$: surjectivity follows. To prove bijectivity,
we note that  the operator $\pi_{\maclKzo }{\mathcal
L}_{\phi,\psi}$ is injective: indeed, if $u\in{\maclKzo }$ is
in the kernel of $\pi_{\maclKzo } {\mathcal L}_{\phi,\psi}$, then
\bel{orjus}0=\langle {\mathcal
L}_{\phi,\psi}(u),u\rangle _{\Lpsi} =\langle
\phi^mP^*(u),\phi^mP^*(u)\rangle _{\Lpsi}\;, \ee
so $u=0$ from inequality \eq{broker}.
\end{proof}


\section{Regularity}
From uniform ellipticity of $L=PP^*$ and scaling properties (see appendix), there exists a constant $C$ such that for
all $u$ in $\zH^{2m}_{\phi,\psi}$
satisfying $\mathcal L_{\phi,\psi}(u)\in C^{k,\alpha}_{\phi,\varphi}$ we have $u\in
C^{k+2m,\alpha}_{\phi,\varphi}$
with
\bel{erc}
\|u\|_{C^{k+2m,\alpha}_{\phi,\varphi}}\leq C \left(\| \mathcal L_{\phi,\psi}u\|_{
C^{k,\alpha}_{\phi,\varphi}}
+\|u\|_{\zH^{2m}_{\phi,\psi}}\right)\;.
\ee
We {so} obtain the
\begin{proposition}\label{regul}
If  $\mathcal L_{\phi,\psi}u=f$ with
$f\in
C^{k,\alpha}_{\phi,\varphi}\cap \zH^{0}_{\phi,\psi}\;, $ and $u\in H^{2m}_{\phi,\psi}$
 then
$u\in
C^{k+2m,\alpha}_{\phi,\varphi}$, {so that} $U=\psi^2\phi^{2m}P^*u\in
\psi^2\phi^{m} C^{k+m,\alpha}_{\phi,\varphi}\;.$
\end{proposition}

\begin{remark}
The quantity  $\varphi\phi^m\nabla^{(m)}u$ is bounded. {This implies that} $U$ is bounded by a constant times $\psi^2\varphi^{-1}\phi^m=x^{2(a-n+m)}e^{-s/x}$.
So, when $s>0$, $u$ might blow up at the boundary but $U$ vanishes on it,
and the same is true for the derivatives.
\end{remark}

Choose some $\alpha>0$ and  define the Fr\'echet {space}
$C^{\infty}_{\phi,\varphi}$ as the collection of all functions
or tensor fields which are in $C^{k,\alpha}_{\phi,\varphi}$
whatever $k\in\N$, equipped with the family of semi-norms
$\{\|\cdot\|_{C^{k,\alpha}_{\phi,\varphi}},k\in\N\}$. We then
have:
\begin{corollary}\label{Cinfty}
Under the hypotheses of the  proposition \ref{regul}, if $f\in C^{\infty}_{\phi,\varphi}$, then $u$
is in  $C^{\infty}_{\phi,\varphi}$ so that $ U\in \psi^2\phi^{m}C^{\infty}_{\phi,\varphi}$. In particular if $s>0$ (in the definition
of $\varphi$ and $\psi$) then
$$U \in C^\infty(\overline\Omega)\;, $$ and  $U$  can be smoothly
extended by zero across $\partial \Omega$.
\end{corollary}

\section{Compactly supported solutions}\label{secpuf}
In this section we would like to point out the result about compactly supported solutions
of
\bel{puf}
PU=f
\ee when the source $f$ is of compact support (see \cite{Qiu} for a related result when $P$
is the divergence operator acting on vector fields).

Let $\Omega$ be an open set of $M$ with compact closure and smooth boundary. Let $f\in C^\infty(\Omega)$ be a  source with
compact support in $\Omega$. We want to find
a   solution $U\in C^\infty(\overline\Omega)$ of \eq{puf}, vanishing {  at} any order on $\partial\Omega$.
In particular, $U$ can be smoothly extended by zero across $\partial \Omega$.
We assume that
\bel{CNf}
\int_{\Omega}\langle v,f\rangle=0,
\ee
for all $v\in \ker P^*$.

This is an obvious necessary assumption.
\begin{theorem}\label{theorpuf}
If $P^*$ satisfies (API) for some $s>0$ and (KRC) holds then  there exists a solution $U\in C^\infty(\overline \Omega)$ of \eq{puf}, which vanishes {  at} any
order on $\partial\Omega$.
\end{theorem}
\begin{proof}
By the theorem \ref{iso}, there {exists} $u\in{\maclKzo }\cap  \Hpsi{k+2m}$ such that
\bel{solproj}\pi_{\maclKzo }[\psi^{-2}P(\psi^2\phi^{2m}P^*u)-\psi^{-2}f]=0.\ee
Let $(v_i)_{i\in I}$ be a $L^2_\psi$ orthonormal basis of the finite dimensional space $\mathcal K=\;$ker$\;P^*$.
By the proposition \ref{RI2}, $v_i\in H^\infty_{\phi,\psi}$. The proposition \ref{regul} then gives
$v_i\in C^{\infty}_{\phi,\varphi}$.
Now we rewrite the equation \eq{solproj}:
$$
\psi^{-2}PU-\psi^{-2}f-\sum_{i\in I}\left(\langle \psi^{-2}PU-\psi^{-2}f, v_i\rangle_{L^2_\psi}\right)v_i=0,
$$
where $U=\psi^2\phi^{2m}P^*u$.
Hence (recall $f\in C^\infty_c(\Omega)$)
$$\mathcal L_{\phi,\psi}u=\psi^{-2}f+\sum_{i\in I}\left(\langle P U-f,v_i\rangle_{L^2}\right)v_i\in C^{\infty}_{\phi,\varphi}.$$
Therefore, by the corollary \ref{Cinfty}, $U\in C^{\infty}(\overline \Omega),$ and $U$ vanishes at any order on $\partial\Omega$.

Let us show that $\pi_{\mathcal K }[\psi^{-2}P(\psi^2\phi^{2m}P^*u)-\psi^{-2}f]=0$.
For all  $v\in \mathcal K$,
\begin{eqnarray*}
\langle \psi^{-2}P(U),v\rangle_{\Lpsi(\Omega)}&=&\langle P(U),v\rangle_{L^2(\Omega)}\\
&=&\langle U,P^*v\rangle_{L^2(\Omega)}+\int_{\partial\Omega}B(U,v)
\\
&=&\int_{\partial\Omega}B(U,v)=0,
\end{eqnarray*}
where $B$ is  a bilinear  $(m-1)$-order operator  appearing after $m$ integrations by parts.
Finally, from the condition \eq{CNf}, for all $v\in \mathcal K$,
$$
\langle \psi^{-2}f,v\rangle_{\Lpsi(\Omega)}=\langle f,v\rangle_{L^2(\Omega)}=0.
$$
\end{proof}

\section{The gluing}
Let $V$, $W$, $\Omega_i$, $\Omega$, $\chi$ be as in the introduction of the paper (see also figure A there) and let
$$T=\chi V+(1-\chi) W.$$
We work on the open set $\Omega$. {Unless otherwise specified}, all the spaces are understood on that open set.
By construction, $T$ is equal to $V$ near $\partial_1\Omega$ and to $W$ near $\partial_2\Omega$, so {that} $\psi^{-2}PT=0$ near these boundaries.
{In particular}, $\psi^{-2}PT$ is in any weighted space {introduced in this paper}.

We have the
\begin{theorem}\label{thetheorem2}
Let $k\geq [\frac n 2]+1$.
If $V\in C^{k+m,\alpha}(\Omega_1)$, $W\in C^{k+m,\alpha}(\Omega_2)$, (API)  is satisfied  for {some} $s>0$.
and (FC) holds, then there {exists} $u\in{\maclKzo }\cap  \Hpsi{k+2m}\cap C^{k+2m,\alpha}_{\phi,\varphi}$ such that $P(T+U)=0$,
where $$U=\psi^2\phi^{2m}P^*u\in \phi^m\psi^2C^{k+m,\alpha}_{\phi,\varphi}\subset C^{k+m,\alpha}(\overline \Omega)$$
can be $C^{k+m,\alpha}$ extended by zero across $\partial\Omega$.
\end{theorem}
\begin{proof}
The proof is the same as the proof of theorem \ref{theorpuf} with $f=-P(T)$.
We only need to verify that $f$ is $L^2$ orthogonal to the kernel of $P^*$.
For all  $v\in \mathcal K$,
\begin{eqnarray*}
\langle P T,v\rangle_{L^2(\Omega)}
&=&\langle T,P^*v\rangle_{L^2(\Omega)}+\int_{\partial_2\Omega}B(T,v)-\int_{\partial_1\Omega}B(T,v),\\
&=&\int_{\partial_2\Omega}B(W,v)
-\int_{\partial_1\Omega}B(V,v)=0,
\end{eqnarray*}
where $B$ is  a bilinear  $(m-1)$-order operator  appearing after $m$ integrations by parts.
We finally apply  the proposition \ref{regul} to get the desired regularity.
\end{proof}

\begin{remark} When $V\in C^m(\overline\Omega)$ one has
$$
\int_{\partial_2\Omega}B(V,v)-\int_{\partial_1\Omega}B(V,v)=\int_{\partial\Omega}B(V,v)=\langle P(V),v\rangle_{L^2(\Omega)}-\langle V,P^*v\rangle_{L^2(\Omega)}=0.
$$
Thus in the definition of (FC), one can replace the integral of $B(V,v)$ on $\partial_1\Omega$
by the integral of $B(V,v)$ on $\partial_2\Omega$. Of course the same substitution can be done if $W\in C^m(\overline\Omega)$.
\end{remark}

\begin{remark}\label{collegene}
The gluing procedure described  above can also be used to solve the more general equation
$$
P(\chi V+(1-\chi)W+U)=\chi P(V)+(1-\chi)PW,
$$
Such a generalization is interesting when a bound on the image has to be respected
(see eg. \cite{Delay:Collescal}).
\end{remark}
\begin{remark}
If the flux of $V$ on $\partial_1\Omega $ is zero, then one can glue $V$ with $W=0$. This allows one to truncate a solution
or to make vanish a solution on a chosen region. In particular one can construct solutions on quotients or on connected sums.
This has also its utility to prove density of compactly supported elements in the kernel of $P$.
\end{remark}

\section{Infinite dimensional kernel}
We assume that the open set $\Omega$ admits a small open ball $B$ where (API) and (KRC) hold.
This is the case in all the applications of section \ref{examples} where (API) and (KRC) always hold
on any smooth relatively compact open set.
Let us show that  the set of smooth compactly supported elements
in the kernel of $P$ on $B$ (then on $\Omega$) is infinite dimensional.
We may assume that $\Omega=B$ and that it is a small open ball in a compact riemannian manifold $M$.
From \cite{BEM}, the set of elements in ker$P\cap \mathring H^m(M)$, which is compactly supported in $\overline \Omega$, is infinite dimensional.
We choose a non-trivial element $U$ in this set.
 Let $(U_i)$ be sequence of smooth sections , compactly supported in $\Omega$ (it is not needed here),
 such that $U_i$
tends to $U$ in $H^m$.
Since the symbol of the operator $P$ is surjective, we have  (see eg. \cite{BergerEbin}, \cite{Ebin1970}, \cite{KohnNirenberg})
$$
H^m(M)=\ker P\oplus\mbox{Im} P^*,
$$
where the sum is $L^2$ orthogonal.
Thus $U_i=P^*u_i+V_i$ where $u_i\in H^{2m}(M)$ and $V_i\in ker P$.
Now $PP^*u_i=PU_i$ is smooth and $PP^*$ is elliptic so that $u_i$ is smooth and the same is true for $V_i$. Since $U_i$ tends to $U$ in ($H^m$ then in) $L^2$ and
$$\|U_i-U\|^2_{L^2(M)}=\|P^*u_i\|^2_{L^2(M)}+\|V_i-U\|^2_{L^2(M)}\geq \|V_i-U\|^2_{L^2(\Omega)},$$
$V_i$ tends to $U$ in $L^2(\Omega)$. In particular there exists $i$ such that $V:=V_i$
is non trivial on $\Omega$. We thus have a smooth non trivial element $V$ in the kernel of $P$
on $\Omega$. It suffice now to glue $V$ with $0$ near the sphere boundary.
As this procedure is valid for any such $U$, the conclusion follows.

Intuitively the result may be true without the conditions (API) and (KRC):
we will study this question in  the future.

\section{The basic example}\label{secex}
Let $P$ be a linear operator of order $m>0$, with smooth coefficients on $\overline\Omega$, such that
$$
P^*u=A(\nabla^{(m)}u)+\mbox{lot},
$$
where $A$ is an injective linear operator with smooth coefficients {   up to the boundary}.

\begin{lemma}
With the notations above: the operator $P^*$ satisfies (API).
\end{lemma}
\begin{proof}
We have
$$
\phi^mP^*u=A(\phi^m\nabla^{(m)}u)+\phi(\phi^{m-1}\mbox{lot}).
$$
As $\phi$ goes to zero near the boundary, for any $\epsilon>0$, if $u$ has compact support
sufficiently close to $\partial\Omega$ then
$$
|\phi(\phi^{m-1}\mbox{lot})|_{L^2_\psi}\leq \epsilon |u|_{H^{m-1}_{\phi,\psi}}.
$$
On the other hand, from the hypothesis on $A$, there exists $c>0$ such that
$$
|A(\phi^m\nabla^{(m)}u)|_{L^2_\psi}\geq c|(\phi^m\nabla^{(m)}u)|_{L^2_\psi}.
$$
Combining those inequalities with \cite{CD2} proposition C.4
used $m$ times, we thus obtain  the estimate near the boundary
$$
||\phi^mP^*u||_{L^2_\psi}\geq C||u||_{H^{(m-1)}_{\phi,\psi}}.
$$
This last inequality clearly implies (API).

\end{proof}

%

\begin{lemma}\label{pkrc}
The operator $P^*$ satisfies (KRC).
\end{lemma}
\begin{proof}
We work on a coordinate system $(x^1,...,x^n)$ near a point $p$ on the boundary. Thus we can adopt the following assumptions:
$\Omega=(-1,1)^{n-1}\times(0,2)$,  $\partial\Omega=\{x^n=0\}$, $p=0$, and $u\in C^\infty(\Omega,\R^N)$.
We consider the family  of paths $\gamma_x(t)$=$(x,0)+(0,...,0,1-t)$ where $x$
is close to zero in $(-1,1)^{n-1}$, and $t\in[0,1]$.
The (system of) equation $P^*u=0$ can be written
$$
\partial_{i_1}...\partial_{i_m}u^i+\mbox{lot}=0.
$$
This is standard to transform this partial  differential system to a first order one by introducing the derivatives of $u$ as new functions
$V=(u,\partial u,...,\partial^{(m-1)}u)$ and then  transform the system above to a first order system
$$
\partial_j V^i+A_{kj}^iV^k=0,
$$
where $V\in C^\infty(\Omega,\R^{N'})$.
Let us define $f^i_x(t)=V^i(\gamma_x(t))$. The functions $f^i_x$ satisfy the linear ordinary differential  system (note $(\gamma^j_x)'=-\delta^j_n$)
$$
(f^i_x)'-A_{kn}^if^k_x=0,
$$
with  coefficients  depending smoothly on $x$ and $t\in[0,1]$.
Classical results about  ordinary differential system show that $f_x(t)$ is well defined for all $t\in[0,1]$ and depends smoothly on $x$
and $t\in[0,1]$, so that $V$ and then $u$
are smooth near $p$.
\end{proof}

\begin{remark}\label{remext}
Each time that it is possible to rewrite the solutions of $P^*u=0$ to a first order system as in the preceding proof,
the solutions will be smooth {  up to the boundary}. Thus (KRC) holds also for other
natural geometric operators (see \cite{BCEG}).
\end{remark}
We now point out two  geometric operators defined in section \ref{examples}.
\begin{corollary}
The Killing operator and the conformal Killing operator satisfy (KRC).
\end{corollary}
\begin{proof}
One rewrites the conformal Killing equation to a first order system (see eg. \cite{BCEG})
and use the remark \ref{remext}.
One can also use the fact that if $X$ is a conformal Killing  vector field then (see eg. \cite{ChristodoulouMurchboost})
$$
\nabla^{(3)}X+R_0\bullet\nabla X+R_1\bullet X=0,
$$
where $R_0$ and $R_1$ are linear expressions in Riem$(g)$ and $\nabla$Riem$(g)$ respectively.
{   The same can be done for the Killing operator}.
\end{proof}

\begin{lemma}
On any connected component of $\Omega$, the dimension of the kernel of $P^*$ does not exceed
the number of components of derivatives of $u$ of order less or equal to $m-1$.
\end{lemma}
\begin{proof}
One can assume that $\Omega$ is connected. The proof is the same as the proof of lemma \ref{pkrc} except
that $p$ is now an interior point and $\gamma_x$ is a ray emanating from $p$. Thus $u$ is determined
around $p$ by its values {   with all of its derivatives of order less or equal than $m-1$ at $p$}.
The dimension of the (local) kernel of $P^*$ is then bounded by a uniform constant,
so this is also true for the kernel of $P^*$.
\end{proof}

%


\section{Applications}\label{examples}
\subsection{Divergence free vector fields}\label{secdivfree}
By identifyibg {vector fields with forms}, we consider $P=d^*$, the divergence operator from one forms to functions:
$$
d^*\omega=-\nabla^i\omega_i,
$$
Elements in the kernel of $P$ are naturally studied in a lot of physics {contexts} such as fluid mechanics or electromagnetism (see \cite{Galdi}, \cite{WangYang} for instance).
In fact divergence free fields (also called  solenoidal, or incompressible, or transverse, depending on the setting) have the
nice property that their flow preserves the volume of any domain.

The formal $L^2$ adjoint of $P$  {is} $P^*=d$, the differential on functions.
The kernel of $d$ is the set of constant functions so  (KRC) {holds}.  
{The (API)  is proved in \cite{CD2}[Proposition C.4. page 75]}.


Let us give an application on $\R^n $ to the case where   the vectors {fields} are divergence free (and/or regular) only outside
a compact set $K$ as in electricity or newtonian gravity for instance. In this case one can take two conditionally compact open set $O_i$'s
such that $K\subset O_1\subset\overline O_1\subset O_2$ and  define $\Omega_1=O_2\backslash K$ and $\Omega_2=\R^n\backslash \overline O_1$.
The two vector fields can be glued as before up to the kernel.
The kernel projection {corresponds} to the difference of their respective flux across, say $\partial_2\Omega$: it is trivial if they have
the same flux.

For example in $\R^3$, we {can glue} any electric field $E$  with vanishing electric density ($\rho=\div E$) outside $K$
and with total charge $Q$, with the electric field surrounding a point charge  given by Coulomb's law:
$$
E_Q=\frac1{4\pi}\frac Q{r^2}\frac{\vec{r}}r,
$$
where $r=\sqrt{x^2+y^2+z^2}$ and $\vec{r}=(x,y,z)$.
This {gives} a model with the same interior (i.e. on $O_1$) field and very simple infinity.

\begin{remark}
On can also imagine a more sophisticated gluing (and/or extension), using open sets as those appearing in \cite{KMPT}
for instance.
\end{remark}

\begin{remark}
Here the gluing result of $V$ and $W$ can be trivially done if both $V$ and $W$ are coexact.
\end{remark}

\begin{remark}
Hodge duality provides an easy translation from our result about divergence free one forms to a result about closed $(n-1)$-forms.
\end{remark}

\begin{remark} \label{contrex} It is tempting to generalize to the following  Hodge-De Rahm type operator on k-forms:
Consider the operator $P^*$ from $k$ {forms} to $k+1$ forms times $k-1$ forms defined by
$$P^*(\omega)=(d\omega, d^* \omega).$$
Then $P(\alpha,\beta)=d^*\alpha+d\beta$. Note that $PP^*=dd^*+d^*d$ is the Hodge-De Rham Laplacian.

Since the  symbol of $PP^*$ is bijective, the symbol of $P^*$ is injective.
If $n\geq 3$ and $1\leq k\leq n-1$ the symbol of $P^*$ is  not surjective because
of the dimensions of the fibers: indeed recall that the  dimension of the fiber of $\Lambda^k$ is $\displaystyle{(\begin{array}{c}n\\k\end{array})}$.

The kernel of $P^*$ is related to the Hodge cohomology but without boundary conditions:
(KRC) is not satisfied in this context.

To be more explicit, let us illustrate with an example on $\R^n$ ($n\geq 3$):
Let $\Omega$ be the  ball of center $(1,0,...,0)$ and radius $1$.
Let $v=r^{2-n}$ be (a constant time) the fundamental solution of the Laplacian and define $u:=dv$.
Clearly $v$ is in the kernel of $P^*$ acting on 1-forms on $\Omega$ and $v$
is not continuous on $\partial \Omega$.

Note also that the (API) condition is  not satisfied here because it implies the finite dimension
of the kernel of $P^*$ whereas here this kernel is infinite-dimensional. In fact, as in the preceding example,
{  one can consider for any reasonable function (or measure) $h$  on the boundary, the solution $u_h$ of the Dirichlet problem : $\Delta u=0$ on $\Omega$
with $u=h$ on $\partial\Omega$. For any such $h$,  we can define  $v_h=du_h$ in the kernel of $P^*$ acting on 1-form on $\Omega$}.

\end{remark}
\subsection{(Multi-)divergence free tensors}
More generally, let us consider the divergence operator $P=\div$, acting from rank $r+1$ covariant tensor fields to
rank $r$ covariant tensor fields :
$$(\div u)_{i_1...i_r}=-\nabla^iu_{ii_1...i_r},$$
its formal $L^2$ adjoint being $P^*=\nabla$, the covariant derivative.
The kernel of $P^*$ {consists of} the parallel rank $r$ tensor fields.
Note that $PP^*$ is the rough Laplacian.
Here again, 
(API) {holds} from \cite{CD2}.
(KRC)  {holds} from section \ref{secex}.
\begin{remark}
We can also consider  the multiple divergence operator from rank $r+m$ covariant tensor fields to
rank $r$ covariant tensor fields:
$$
(\div^{(m)}u)_{i_1...i_r}:=(-1)^m\nabla^{j_m}...\nabla^{j_1}u_{j_1...j_m i_1...j_r}.
$$
The adjoint is $\nabla^{(m)}$ the m-covariant derivative. Here again 
(API) holds from \cite{CD2} proposition C.4
used $m$ times.
(KRC)  {holds} from section \ref{secex}.
\end{remark}
\subsection{Divergence free symmetric  two tensors}
We can also consider the divergence operator $P=\div$, acting from symmetric covariant tensor fields to
one forms :
$$(\div u)_{j}=-\nabla^iu_{ij},$$
its formal $L^2$ adjoint being $(P^*\omega)_{ij}=\nabla_i\omega_j+\nabla_j\omega_i$, the Killing operator.
Elements in the kernel of $P^*$ are one forms associated to Killing vector fields.
Note here that equation \eq{RIC} is also called the   (weighted) Korn inequality used in elasticity theory (see eg. \cite{DainKorn}).
(API)  is already proven in \cite{CD2} where $P^*$ is called $S$ there.
(KRC)  {holds} from section \ref{secex}.
\begin{remark}
The operator $\div$ can be replaced by $Pu=\div u+c\;d\Tr u$, for any constant $c\neq\frac1n$, such as
the Bianchi operator $c=\frac12$ (elements in the kernel of the Bianchi operator are called harmonic tensors \cite{ChenNagano})
or the momentum  constraint operator $c=1$ . In such a case the kernel
of $P^*$ is the Killings.
\end{remark}

\subsection{TT-tensors}\label{secTT}
 TT-tensors are  trace free and divergence free symmetric two tensors.
They have the following conformally covariant property: if $V$ is a TT tensor for $g$ and $u$
is a positive function, then $u^{-2}V$ is a TT tensor for $u^{4/(n-2)}g$.
Construction of such a tensor arises  when studying the constraint equation in general relativity \cite{Carges}. In some situations, {it} is  important to construct compactly supported tensors as
it as been done on $\R^3$ in \cite{CorvinoAS} using explicit formulas  and on $\R^n$ in \cite{Gicquaud3}
using the Fourier transform.
Also, when doing {the} gluing procedure, it is important to truncate  the TT-tensor on a small ball (see eg. \cite{IMP}).
The procedure described here gives a construction on {\it any} Riemannian manifold.

Here we consider $P=\mathcal D^*$, the divergence operator from trace free symmetric two tensors to one forms:
$$
(\mathcal D^* u)_i=-\nabla^ju_{ij},
$$
its formal $L^2$ adjoint being $P^*=\mathcal D$, the conformal Killing operator, also called
the Ahlfors operator:
$$
(\mathcal D\omega)_{ij}=\frac12(\nabla_i\omega_j+\nabla_j\omega_i)+\frac 1 n d^*\omega \;g_{ij}.
$$
Note that $\mathcal D^*\mathcal D$ is usually called the vector Laplacian.
Elements in the kernel of $\mathcal D$ are one forms corresponding to conformal Killing fields.
(KRC)  {holds} from section \ref{secex}.
Note that in \cite{BCS} it is shown that generically there {does} not exist local nor global non trivial conformal Killing fields.  Also, on $\R^n$ with $n\geq 3$,  the space of conformal
Killing is explicit  and has dimension  $(n+1)(n+2)/2$. For $n=2$ the operator $\mathcal D^*$ is (determined) elliptic
 and can not verify our hypothesis. For instance on any open set of $\R^2$, any analytic function $F(x,y)=a(x,y)+ib(x,y)$
gives rise to a conformal Killing form $\omega=a dx+bdy$ and reciprocally, so neither  (KRC) nor (API) are true as in
remark \ref{contrex}.

(API) 
is not already  been proven in the literature in this context and some more work is needed here
(see however \cite{AndElli} for related results).
%

\begin{proposition}
\label{francois}
When $n\geq3$, the operator $\mathcal D$ satisfies (API).
\end{proposition}
\begin{proof}
Here we can not  apply  \cite{CD2}, corollary D.5 page 79, because we need the same kind of inequality with
the Ahlfors operator in place of the Killing one, and $|\mathcal D\omega|$ does not control $|d^*\omega|$ pointwise.
Thus we need a more precise estimate.

Before going to the proof, let us make some simplifying assumptions. As we will work near the boundary, we may
first choose for $x$ the distance to the boundary. The metric then take the form $g=dx^2+h(x)$, where $h(x)$ is
a family  of metrics on $\partial\Omega$. The difference between the connection of $g$ and the connection of $dx^2+h(0)$
goes to zero on the boundary.
We then may assume that the metric is a product $g=dx^2+h$, where $h$ is a fixed metric
on $\partial\Omega$ and $x\in(0,\epsilon)$, with a small $\epsilon$.

For a one form $\omega$ compactly supported in $\Omega_\epsilon:=(0,\epsilon)\times\partial\Omega$, we decompose
$$
\omega=fdx+\alpha,
$$
where $f$ is a function on $\Omega_\epsilon$ and $\alpha $ a one form in $C^\infty(\Omega_\epsilon,T^*\partial\Omega)$.
First, from the equality
$$
\langle \mathcal D\omega,fdx\otimes dx\rangle_g=f\partial_xf+\frac1n (d^*\omega)f=(1-\frac1n)f\partial_xf+\frac1n(d^*_h\alpha)f,
$$
combined with
\begin{eqnarray*}
J:=\int x^{2t-2}e^{-2s/x}(-\partial_xf)f&=&\int_{\partial\Omega}\left(\int_0^\epsilon x^{2t-2}e^{-2s/x}[-\frac12\partial_x(f^2)]dx\right)d\mu_h\\
&=&\int_{\partial\Omega}\left(\int_0^\epsilon (s+o(1))x^{2t-4}e^{-2s/x}f^2dx\right)d\mu_h\\
&=&\int(s+o(1))x^{2t-4}e^{-2s/x}f^2,
\end{eqnarray*}
we deduce
\bel{estif}
\int x^{2t-2}e^{-2s/x}[(d^*_h\alpha)f-n\langle\mathcal D\omega,fdx\otimes dx\rangle_g]=(n-1)\int x^{2t-4}e^{-2s/x}(s+o(1))f^2.
\ee
Now, let  us compute
\begin{eqnarray*}
I:=\int x^{2t-2}e^{-2s/x}(d^*_h\alpha)f&=&\int_0^\epsilon x^{2t-2}e^{-2s/x}\left(\int_{\partial\Omega} (d^*_h\alpha)fd\mu_h\right)dx\\
&=&\int_0^\epsilon x^{2t-2}e^{-2s/x}\left(\int_{\partial\Omega} \langle \alpha, d_h f\rangle_hd\mu_h\right)dx.\\
\end{eqnarray*}
Let $(x^i)=(x^0=x,x^A)$ be a coordinate system of $\Omega_\epsilon$ adapted to its character. We rewrite
$$
\langle \alpha, d_h f\rangle_h=\omega^A\partial_A\omega_0=2\omega^A(\mathcal D\omega)_{0A}-\omega^A\partial_0\omega_A=
2(\mathcal D\omega)(\alpha,dx)-\frac12\partial_x|\alpha|^2.
$$
This shows that
$$
I=2\int x^{2t-2}e^{-2s/x}\langle\mathcal D\omega,\alpha\otimes dx\rangle_g+\int x^{2t-4}e^{-2s/x}(s+o(1))|\alpha|^2
$$
so
\bel{estialpha}
\int x^{2t-2}e^{-2s/x}[(d^*_h\alpha)f-2\langle \mathcal D\omega,\alpha\otimes dx\rangle_g]=\int x^{2t-4}e^{-2s/x}(s+o(1))|\alpha|^2.
\ee

We will now show that any term appearing in the left hand-side  of \eq{estialpha} and \eq{estif} can be estimated in absolute
value by a term of the form $\frac a 2\|\phi\mathcal D\omega\|^2_{L^2_\psi}+\frac1{2a}\|\omega\|^2_{L^2_\psi}$, for any $a>0$, possibly modulo terms
of the form $\|o(1)\omega\|^2_{L^2_\psi}$, where $o(1)\longrightarrow 0$ when $x$ goes to zero. This will then prove the announced result.
Let
\begin{eqnarray*}
\eta_{AB}:=(\mathcal D\omega)_{AB}&=&\frac12(\nabla_A\omega_B+\nabla_B\omega_A)+\frac1n(d^*\omega)h_{AB}\\
&=&(\mathcal D_h\alpha)_{AB}+[(\frac1n-\frac1{n-1})d^*_h\alpha-\frac1n\partial_xf]h_{AB}\\
&=&(\mathcal D_h\alpha)_{AB}-\frac1{n-1}(\mathcal D\omega)_{00}h_{AB},
\end{eqnarray*}
then
$$
|\eta|_h^2=|\mathcal D_h\alpha|_h^2+\frac 1{n-1}[(\mathcal D\omega)_{00}]^2.
$$
The inequality $|\nabla\alpha|_h^2\geq \frac1{n-1}|d_h^*\alpha|^2$ then gives
\begin{eqnarray*}
2\int_{\partial\Omega}|\mathcal D_h\alpha|_h^2&=&\int_{\partial\Omega}[|\nabla_h\alpha|^2+\frac{n-3}{n-1}|d_h^*\alpha|^2-\Ric_h(\alpha,\alpha)]\\
&\geq& \int_{\partial\Omega}[\frac{n-2}{n-1}|d_h^*\alpha|^2-\Ric_h(\alpha,\alpha)],
\end{eqnarray*}

Therefore, for any constant $a>0$:
\begin{eqnarray*}
\left|\int x^{2t-2}e^{-2s/x}(d^*_h\alpha)f\right|&\leq& \frac a 2\int x^{2t}e^{-2s/x}(d^*_h\alpha)^2+\frac1{2a}\int x^{2t-4}e^{-2s/x}f^2\\
&\leq&\frac{n-1}{n-2}\frac a 2\int x^{2t}e^{-2s/x}\left[2|\mathcal D_h\alpha|_h^2+O(1)|\alpha|^2\right]\\
&&\;\;\;+\frac1{2a}\int x^{2t-4}e^{-2s/x}f^2\\
&\leq&\frac{n-1}{n-2}\frac a 2\int x^{2t}e^{-2s/x}\left[2|\eta|_h^2+O(1)|\alpha|^2\right]\\
&&\;\;\;+\frac1{2a}\int x^{2t-4}e^{-2s/x}f^2
\end{eqnarray*}
Similarly for any constant $b>0$:
$$
\left|\int x^{2t-2}e^{-2s/x}\langle \mathcal D\omega,\alpha\otimes dx\rangle_g\right|\leq\frac b 2 \int x^{2t}e^{-2s/x}|\nu|_h^2+\frac1{2b}\int x^{2t-4}e^{-2s/x}|\alpha|^2,
$$
where $\nu_A:=(\mathcal D\omega)_{0A}$.
Also, for any constant $c>0$:
$$
\left|\int x^{2t-2}e^{-2s/x}\langle \mathcal D\omega,fdx\otimes dx\rangle_g\right|\leq\frac c 2 \int x^{2t}e^{-2s/x}|(\mathcal D\omega)_{00}|^2+\frac1{2c}\int x^{2t-4}e^{-2s/x}f^2
$$
By combining the last three inequalities for large $a,b,c$ with equations \eq{estialpha},\eq{estif} and the fact that
$$
|\mathcal D\omega|^2=|(\mathcal D\omega)_{00}|^2+2|\nu|^2_h+|\eta|^2_h
$$
we get the proof of the proposition \ref{francois}.
\end{proof}

As in the case of electric fields, one can use the same procedure  to glue any TT-tensor $V_{ij}$ defined outside a compact  set $K$ of $\R^3$
(containing zero for simplicity)
with a unique Beig-Bowen-York tensor \cite{BeigBY}\cite{BeigKrammer} as follows.
Let us consider the $10$ parameters family of Beig-Bowen-York tensors
$$E_{ij}=^1K_{ij}+...+^4K_{ij},$$
where the $^lK_{ij}$'s are defined in \cite{BeigBY}.
Let also consider a fixed basis $(v_1,...,v_{10})$ of the space of conformal Killing fields (choose  particular $^l\eta_j(x)$'s in \cite{BeigBY}
for instance).
Let the $\Omega_i$'s be chosen
as in section \ref{secdivfree}. The two TT-tensors fields $V$ and $E$ can be glued on $\Omega$ modulo kernel.
For the kernel projection, one projects on any elements $v_i$ of the basis: each of them gives the difference of their
respective "flux"\footnote{In this setting this correspond to linear momentum, angular momentum,...}  across, say $\partial_2\Omega$ :
$$
p_i:=\int_{\Omega}\langle\mathcal D^* T,v_i\rangle=\int_{\partial\Omega_2}E(v_i,\eta)-\int_{\partial\Omega_2}V(v_i,\eta),
$$
where $\eta$ is the unit normal.
One then chooses the $10$-parameters of $E$ to make vanish the $10$ projections (it is easily
be checked that the linear map which sends  the $10$ parameters of $E$ to the $10$ real numbers $\int_{\partial_2\Omega}E(v_i,\eta)$
is an isomorphism of $\R^{10}$).

This construction has the advantage to produce an infinite dimensional family of  TT-tensors with well know  infinity. Moreover, by using the Licherowicz-York
method (see eg. \cite{IsenbergCMC}) it gives rise
to conformally euclidian CMC initial data  for the Einstein equation. Because of conformal euclidian setting, such a kind of data
is {  valued}  by numerical relativity.

\begin{remark}
Here also one can imagine a more sophisticated gluing (and/or extension), using open sets of the form used in \cite{KMPT}
for instance.
\end{remark}

\begin{remark}
It is  easy to prove  that on any  relatively compact open set of $(M,g)$, the set of smooth TT-tensors
is infinite dimensional. Taking any (small) ball  and gluing TT-tensors with zero as before near the sphere boundary, one deduces that
 the set of smooth TT-tensors with compact support
in a fixed  ball (thus on any open set)  is also infinite dimensional.
\end{remark}

%

\subsection{Linearized scalar curvature operator}
We consider the operator which to a Riemannian metric $g$ give{s} its scalar curvature. The linearization of this operator
is another operator $P$ from symmetric two tensors to functions given by
$$
P h=\nabla^k\nabla^lh_{kl}-\nabla^k\nabla_k(\Tr
\;h)-R^{kl}h_{kl},
$$
Its formal adjoint is
$$
P^*f=\nabla\nabla f-\nabla^k\nabla_k fg-f\;Ric(g).
$$
Those operators were studied by Fischer and Marsden \cite{FM1974}.
It is well know that the dimension of the kernel of $P^*$ is less or equal to $(n+1)$ (see eg. \cite{Corvino}).

Compactly supported elements in the kernel of $P$ {play} an important role
in some situations (see \cite{CorvinoAS} on $\R^n$).
Here again our procedure give{s} a construction in a general context.
(API) is proved in \cite{CD2} and
(KRC)  {holds} from section \ref{secex}. Note that the kernel is trivial
 in generic situations or on small balls \cite{BCS}.

Here also, on $\R^n$ ($n\geq 3$) for instance, one can glue any element in the kernel of $P$, smooth outside a compact set of $\R^n$ with
an element of the family
$$
E=\frac{m}{|x-c|^{n-2}}\mbox{euc},
$$
where euc is the euclidian metric, $m\in\R$, and $c\in\R^n$.\\

\subsection{A non linear application}
As already written in the introduction, the gluing procedure has been used in a non linear context in general relativity.
We are interested here in a non linear operator which appeared  in riemannian  Weyl structures.
For a riemannian manifold $(M,g)$, we consider the operator from one forms to functions defined
by
$$
\mathcal P\theta:=d^*\theta+\frac{n-2}4|\theta|^2.
$$
This operator is related to the scalar curvature of a  Weyl structure by (see \cite{GauduchonWeyl} for instance, with
a different normalisation of $\theta$).
$$
R^W=R(g)+(n-1)\mathcal P\theta.
$$
The linearization  of $\mathcal P$ at $\theta$ is
$$
P\omega=d^*\omega+\frac{n-2}2\langle \theta,\omega\rangle.
$$
The adjoint of $P$ is then
$$
P^*u=du+\frac{n-2}2u\theta.
$$
Let $u$ be in the kernel of $P^*$. If $u$ vanishes at some point $p$, then $u$ vanishes near $p$, and where $u$ does not
vanishes,
$\theta=-\frac2{n-2}d\ln(|u|)$.
Thus the kernel of $P^*$ on any connected open set is trivial if and only if $\theta$ is not exact on this set.
Otherwise the kernel is one dimensional.

One then proceeds as in  \cite{Corvino} to
show that for any smooth function $\rho$ with compact support, and close to zero, there exists a small, smooth one form $U$, with compact
support close to the support of $\rho$, such that up to kernel projection if any,
$$
\mathcal P(\theta+U)=\mathcal P(\theta)+\rho.
$$
In the same way, as in \cite{Delay:Collescal}, exploiting the fact that the norm of the  inverse of the operator $\pi_{\maclKzo } {\mathcal L}_{\phi,\psi}$ in Theorem \ref{iso} is uniformly
bounded for any $\theta'$  close to $\theta$ in $W^{k+1,\infty}_\phi$, one can glue two Weyl form connexions close one to each other on a compact region,
by interpolating their images with $\mathcal P$.
$$
\mathcal P[\chi\theta+(1-\chi)\theta'+U]=\chi\mathcal P(\theta)+(1-\chi)\mathcal P(\theta').
$$
As in \cite{Corvino}, on $(\R^n,$euc$)$ it is possible to glue an asymptotically flat Weyl form connexion (see \cite{Vassal} for a definition) such that $\mathcal P(\theta)=0$
with a
$$
\theta_m:= \frac{4m\;dr}{r^{n-1}+mr}=df_m\;,\;\;f_m=-\frac4{n-2}\ln\left(1+\frac m{r^{n-2}}\right),
$$
on an annulus close to infinity, to a form connexion in $\mathcal P^{-1}(\{0\})$.
{  In particular the set of AF Weyl connexions on  $(\R^n, $euc$)$ with vanishing Weyl scalar curvature and correspond to the Levi Civita
connection of a Schwarzschild metric}
$$
g_m=\left(1+\frac m{r^{n-2}}\right)^{\frac4{n-2}}\mbox{euc}
$$ outside of a compact set, is dense in
the set of AF Weyl connexions on  $(\R^n, $euc$)$ with vanishing Weyl scalar curvature.

%
\section{Appendix : Scaling Estimates }\label{Sscaling}
For completeness, we recall the appendix B of \cite{CD2} (with some misprints corrected) and we add
a Sobolev estimate let to the reader there.

\subsection{Preliminary}
The weight functions we consider have the following property:

\begin{equation}\label{lcond}
 |\phi^{i-1}\nabla^{(i)}\phi|_g\leq C_{i}\;,\;\;\;
|\phi^{i}\psi^{-1}\nabla^{(i)}\psi|_g\leq C_{i}\;,
\end{equation}
for $i\in\N$ and for some constants $C_i$. This implies that  for all
$i,k\in \N$ we have \be \label{0.1}
|\phi^{i-k}\nabla^{(i)}\phi^k|_g\leq C_{i,k}\;,\;\;\;
|\phi^{i}\psi^{-k}\nabla^{(i)}\psi^k|_g\leq C_{i,k}\;. \ee Thus, for $m,s,i,k\in\N$  the maps
$$
\psi^{-m}\phi^{i-s}\nabla^{(i)}(\phi^s\psi^m\cdot)
:\zH^{k+i}_{\phi,\psi}\longmapsto \zHkpp \;,
$$
$$
\psi^{-m}\phi^{i-s}\nabla^{(i)}(\phi^s\psi^m\cdot)
:W^{k+i,\infty}_{\phi}\longmapsto W^{k,\infty}_{\phi}\;,
$$
$$
\psi^{-m}\phi^{-s}\nabla^{(i)}(\phi^{i+s}\psi^m\cdot)
:\zH^{k+i}_{\phi,\psi}\longmapsto \zHkpp \;,
$$
\be\label{mpp}
\psi^{-m}\phi^{-s}\nabla^{(i)}(\phi^{i+s}\psi^m\cdot)
:W^{k+i,\infty}_{\phi}\longmapsto W^{k,\infty}_{\phi}\;, \ee are
continuous and bounded.  The function $\varphi$ satisfies the
same condition \eq{0.1} as  $\psi$, so that we can replace
$\zH^j_{\phi,\psi}$ by $C^{j,\alpha}_{\phi,\varphi}$ in the equations \eq{mpp}.

 For all
$p\in \Omega$, we denote by $B_p $, the open ball of center $p$ with
radius $\phi(p) /2$. Changing the defining function $x$ if necessary, we require that for all $p\in \Omega $,
\bel{scalprop0}B(p,\phi(p) ) \subset \Omega\;.\ee

\begin{lemma}\label{supinf}
There exists a constant $C_1>0$ such that for all
$p\in \Omega $ and for all $y\in B_p $, we have
\bel{scalprop1}
C_1^{-1}\phi(p) \leq \phi (y)\leq C_1\phi(p) .
\ee
There exists a constant
$C_2>0$ such that for all $p\in \Omega $ and for all $y\in B_p $, we have
\bel{scalprop2}
C_2^{-1}\varphi(p)\leq \varphi (y)\leq C_2\varphi(p).
\ee
The same assertion holds when substituting $\varphi$ with $\psi$.
\end{lemma}
\begin{proof}

Let us remind that $\phi = x^2$ where  $x$ is
equivalent to $d(\cdot,\partial\Omega)$. In order to prove \eq{scalprop1}, we
compute for all $q\in B_p$: by the triangle inequality,
$$
d(p,\partial\Omega)-d(p,q)<d(q,\partial\Omega)\leq d(p,\partial\Omega)+d(p,q)\;.
$$
Then, since $d(p,q)<x(p)^2/2$ for $q\in B_p$,
$$
d(p,\partial\Omega)-x(p)^2/2<d(q,\partial\Omega)\leq d(p,\partial\Omega)+x(p)^2/2\;.
$$
From \eq{scalprop0} we have $x(p)^2<d(p,\partial\Omega)$. This gives
$$
d(p,\partial\Omega)/2<d(q,\partial\Omega)\leq 3d(p,\partial\Omega)/2\;,
$$
and as $x$ is equivalent to $d(.,\partial\Omega)$ we obtain
\eq{scalprop1}.

 Now recall that $\psi =e^{-s/x}$,
where $s\in\R,\; s\neq 0$.
Moreover, for all $q\in B_p$,
$$
e^{-s/x(p)}e^{s/x(q)}=e^{-s(x(p)-x(q))/x(p)x(q)},
$$
but $|x(p)-x(q)|$  is bounded by some constant times $x(p)^2$ and
$x(p)x(q)$ is equivalent to $x(p)^2$. We so get \eq{scalprop2} for $\psi$.

If $\varphi_1$ and $\varphi_2$ satisfy
\eq{scalprop2}, then $\varphi_1 \varphi_2$ also will. It follows
that $\varphi= x^\alpha e^{s/x}$ can also be used as a weighting
function in our context.
\end{proof}

\subsection{Estimates in H\"older spaces}
 In this section we will see that the choice of functions $\phi $ and $\varphi $ will guarantee
the estimate \eq{erc}.   We assume
that $\Omega$ is an open subset of $\R^n$, and that the
elliptic operator we work with is an operator on functions. The
result generalizes to tensor fields on manifolds by using
coordinate patches, together with covering arguments.

 For $p\in \Omega$, we define
$$
\varphi_p  :B(0,1/2)\ni z\mapsto p+\phi(p) z\in B_p .
$$
For all functions $u$ on $\Omega$ and all multi-indices $\gamma$ we
have
$$
\partial_z^\gamma(u\circ
\varphi_p )=\phi(p) ^{|\gamma|}(\partial^\gamma u)\circ \varphi_p
.
$$
Let $L(p,\partial)$ be a strictly elliptic (\emph{e.g.}, in the
sense of Agmon-Douglis-Nirenberg) operator of order $2m$ on $\Omega$ and
set
$$
(L_\phi u)(p):=[L(\cdot,\phi \partial)u](p).
$$
Notice that in our setting $L_\phi=\mathcal L_{\phi,\psi}$ will be elliptic uniformly
degenerate whenever $\phi(p) $ approaches zero in some regions. We
assume that the coefficients of $L$ are in $C^{k,\alpha}_{\phi,1}
(\Omega)$. For all $p\in \Omega $, we define the elliptic operator $Q_p$ on
$B(0,1/2)$ by
$$
Q_p(z,\partial):=L(\varphi_p (z),(\phi(p) )^{-1}\phi
\circ\varphi_p (z)\partial).
$$
We then have
$$
Q_p(u\circ\varphi_p )=(L_\phi u)\circ\varphi_p .
$$

By the lemma \ref{supinf}, the $C^{k,\alpha}(B(0,1/2))$ norm of the coefficients of
$Q_p$ are bounded by the $C^{k,\alpha}_{\phi,1} (\Omega)$ norm of the
coefficients of $L$. On the other hand, $Q_p$ is strictly elliptic and
by the usual interior elliptic estimates, there exists $C>0$, which
depends neither on $p$ nor on $v$,  such that, for all functions $v\in
L^{2}(B(0,1/2))$ such that $Q_pv$ is in $C^{k,\alpha}(B(0,1/2))$
we have $v\in C^{k+2m,\alpha}(B(0,1/4))$ and
$$
\|v\|_{C^{k+2m,\alpha}(B(0,1/4))}\leq
C(\|Q_pv\|_{C^{k,\alpha}(B(0,1/2))}+\|v\|_{L^{2}(B(0,1/2))}).
$$
So, if $u$ is in $ L^{2}_{\varphi\phi^{-n/2}} (M)$ with $Lu\in
C^{k,\alpha}_{\phi ,\varphi }(\Omega)$, then $u\in
C^{k+2m,\alpha}_{loc}$.

For $p\in \Omega $, we  define $B'_p $ the ball of center $p$ and
radius $(1/4)\phi(p) $. It follows from \eq{scalprop1} that there
is a $p$--independent number $N$ such that each $B_p$ is covered
by $N$ balls $B_{p_i(p)}'$, $i=1,\ldots,N$. We then have (the
second and the last inequalities below come from \eq{scalprop2})

\begin{eqnarray}\nonumber
\|u\|_{C^{k+2m,\alpha}_{\phi ,\varphi }(\Omega)}&\leq& C \sup_{p\in \Omega
}\|u\|_{C^{k+2m,\alpha}_{\phi ,\varphi }(B'_p )}\\\nonumber &\leq&
C \sup_{p\in \Omega }(\varphi(p)\|u\|_{C^{k+2m,\alpha}_{\phi(p) ,1}(B'_p
)})\\\nonumber &\leq& C \sup_{p\in M }(\varphi(p)\|u\circ\varphi_p
\|_{C^{k+2m,\alpha}(\varphi_p ^{-1}(B'_p ))})
\\\nonumber &=& C
\sup_{p\in \Omega }(\varphi(p)\|u\circ\varphi_p
\|_{C^{k+2m,\alpha}(B(0,1/4))})
\\\nonumber &\leq& C \sup_{p\in \Omega
}[\varphi(p)(\|(L_\phi u)\circ\varphi_p \|_{C^{k,\alpha}(B(0,1/2))}
+\|u\circ\varphi_p \|_{L^{2}(B(0,1/2))})]\\\nonumber &\leq& C
[\sup_{p\in \Omega }(\varphi(p)\|L_\phi u\|_{C^{k,\alpha}_{\phi(p) ,1}(B_p
)})
+\sup_{p\in \Omega }(\|u\|_{L^{2}_{\varphi\phi^{-n/2}}(B_p  )})]\\
&\leq& C (\|L_\phi u\|_{C^{k,\alpha}_{\phi ,\varphi
}(M)}+\|u\|_{L^{2}_{\varphi\phi^{-n/2}}(\Omega)})\;.
\label{weighcalc}\end{eqnarray} In particular $u\in
C^{k+m,\alpha}_{\phi ,\varphi }(\Omega)$. An identical calculation
gives
\begin{eqnarray*}
\|u\|_{C^{k+m,\alpha}_{\phi ,\varphi }(\Omega)} &\leq& C
(\|L_\phi u\|_{C^{k,\alpha}_{\phi ,\varphi
}(\Omega)}+\|u\|_{L^{\infty}_{\varphi}(\Omega)})\;.
\end{eqnarray*}

\subsection{Estimates in Sobolev spaces}

The proof of the following lemma is left to the reader in \cite{CD2}. It is somehow reminiscent of \cite{CD5} lemma 3.6.
Here, we recall the proof for completeness.
\begin{lemma}
Let $k\in\N$. There exists a
constant $C$ such that for
all $u\in{H^{k+2m}_{\phi,\psi}}$,
$$
 \|u \|_{H^{k+2m}_{\phi,\psi}} \le C \Big( \|\mathcal L_{\phi,\psi} u \|_{H^{k}_{\phi,\psi}} +
 \|u\|_{H^{0}_{\phi,\psi }}\Big)\;.
$$
\end{lemma}
\proof
Let $f(x)=\frac{\sqrt{1+4(x-x^2)}-1}2$. Let $x_0\in(0,1)$ to be defined later and let us consider the sequence
$x_{k+1}=f(x_k)$. From the definition, we have
$$
x_{k+1}+x_{k+1}^2=x_k-x_k^2,
$$
and  $x_k$ decreases to zero.
As
$$
x-\frac{11}{10}x^2-[f\circ f(x)+\frac{11}{10}(f\circ f(x))^2]=\frac95x^2+O(x^3),
$$
 for $x_0$ small we have
$x_{k+2}+\frac{11}{10}x_{k+2}^2<x_k-\frac{11}{10}x_k^2$.
In particular the number of intervals of the form $J_k=[x_k-\frac{11}{10}x_k^2,x_k+\frac{11}{10}x_k^2]$ that intersect some $J_k$,$k>0$,
is equal to two. At this stage, by  reducing $x_0$ if necessary, we also have
$$
\forall x\in J_k\;,\;\;\;\; \frac12x_k^2<x^2<2x_k^2.
$$
We define $I_k=[x_k-x_k^2,x_k+x_k^2]=[x_{k+1}+x_{k+1}^2,x_k+x_k^2]$.

Let $F_{-1}=\{p\in\Omega, x(p)\geq x_0+x_0^2\}$ and let $F_k=\{p\in\Omega, x(p)\in I_k\}$ for $k\in\N$.

Let $G_{-1}=\{p\in\Omega, x(p)\geq x_0\}$ and let $G_k=\{p\in\Omega, x(p)\in J_k\}$ for $k\in\N$.


  For any
function  $f\in L^1(\Omega)$ we thus get
\bel{firint}
 \int_{\Omega}f = \sum_i \int_{F_i} f
 \;.
\ee
%
From above, for all positive integrable functions $f$ we
have
$$
\int_{F_k}f\leq \int_{G_k}f\leq \sum_{l, F_l\cap G_k\neq\emptyset}\int_{F_l}f.
$$
As the cardinality of $\{l, F_l\cap G_k\neq\emptyset\}$ does not exceed  3, we deduce

\bel{firint2}
  \int_{\Omega}f\le \sum_i \int_{G_i}
  f \le 3  \int_{\Omega}f
 \;.
\ee

The above construction provides a decomposition
of $\Omega$ into "annuli"  $F_k$, the size of which in the
$x$-direction is comparable to $x(p)^2$ for any $p\in F_k$; similarly the sizes of $G_k$ in the
$x$-direction are comparable to $x(p)^2$ for any $p\in G_k$.

We now assume that $x_0$ is small enough to identify $F_k$ with $I_k\times \partial\Omega$ for $k\geq 0$ and the same identifications are possible between $G_k$
and $J_k\times \partial \Omega$.\\

{   We continue with a annulus dependent,
cube decomposition of $\partial \Omega$}, as follows: Let
$\{(\mcO_i,\psi_i)\}_{i=1,\ldots,N}$  be a covering of $\partial \Omega$ by
coordinate charts with each coordinate system $\psi_i^{-1}$
mapping $\mcO_i$ smoothly and diffeomorphically to a
neighbourhood of $[0,1]^{n-1}$; the local coordinates on
$[0,1]^{n-1}$ will be denoted by $\theta^A$. We further assume
that $\cup \psi_i([0,1]^{n-1})$ covers $\partial \Omega$ as well.
Let $\varphi_i$ be an associated decomposition of unity, thus
$\sum_i \varphi_i=1$. Setting $f_i=(\varphi_i f)\circ\psi_i $,
for any integrable function $f$ we have

$$
\int_{[a,b]\times \partial \Omega}f = \sum_{i=1}^N \int_{[a,b]\times
[0,1]^{n-1}} f_i
 \;.
$$

Given  an interval $I_k$ define $m=m(k)\in \N$ by the
inequality
\bel{mineq}
 \frac 1 {m+1} \le x_k^2 < \frac 1 m
 \;.
\ee

Let  $\{K_j\}$ be the collection of closed $(n-1)$-cubes centered at $c_j$, with
pairwise disjoint interiors, and with edges of size $1/m$,
covering $[0,1]^{n-1}$. For any $K_j$ let $\hat K_j$ be the union
of those cubes $K_i$ which have non-empty intersection with $K_j$.
There exists a number $\hat N(n)$ such that $\hat K_j$
consists of at most $\hat N(n)$ cubes $K_i$. It follows that,
for any integrable function $f_i$:
$$
 \int _{[a,b]\times
[0,1]^{n-1}} f_i = \sum _{k}\int_{[a,b]\times {K_k}} f_i
 \;,
$$
and if $f_i\ge 0$ then
$$
 \int _{[a,b]\times
[0,1]^{n-1}} f_i \le  \sum _{k}\int_{[a,b]\times {\hat K_k}} f_i
 \le
 \hat N(n)
  \int _{[a,b]\times
 [0,1]^{n-1}} f_i \;.
$$

 We are ready now to pass to the heart of our argument. Let
$\{\mcU_\ell\}_{\ell\in \N}$ be the collection, without repetitions,
of the sets
$$
 \{ I_k \times \psi_i( K_j)\}_{k\in \N,\ i=1,\ldots,N,\ j=0,\ldots,m^{n-1}}
 \;.
$$
Similarly let $\{\hat\mcU_\ell\}_{\ell \in \N}$ be the collection,
without repetitions, of the sets
$$
 \{J_k \times \psi_i(\hat K_j)\}_{k\in \N,\ i=1,\ldots,N,\ j=0,\ldots,m^{n-1}}
 \;.
$$
{}From  what has been said above we have, for any positive integrable
function $f$,
\beaa
 &
 \displaystyle
 \int_{[0,x_0+x_0^2]\times \partial \Omega} f \le  \sum _\ell \int_{\mcU_\ell}
f
 \le
 N \int_{[0,x_0+x_0^2]\times \partial \Omega} f
 \;,
 &
 \\
 &
 \displaystyle
 \int_{[0,x_0+\frac{11}{10}x_0^2]\times \partial \Omega} f \le  \sum _\ell \int_{\hat\mcU_\ell} f \le N \hat N(n) \int_{[0,x_0+\frac{11}{10}x_0^2]\times \partial \Omega} f
 \;.
 &
\eeaa

If $\mcU_\ell= I_k \times \psi_i(K_j)$, we scale the local coordinates
$(x,\theta^A)$ in $\hat\mcU_\ell$ as
$$
(x,\theta^A)\mapsto ((x-x_k)/x_k^2 ,m (\theta^A-c_j))\;.
$$
This maps all $\mcU_\ell\subset \hat\mcU_\ell$'s to
 fixed cubes
$$
\mcU_\ell \longrightarrow [-1,1]\times [0,1]^n\subset [-\frac{11}{10},\frac{11}{10}]\times
[-1,2]^n \longleftarrow \hat\mcU_\ell
 \;,
$$

By construction there exists a constant $C>0$, independent of $i$,
such that:
$$
  {\sup_{I_i \times \partial \Omega}} \phi\le C {\inf_{I_i \times \partial \Omega}}
  \phi\;, \qquad
  {\sup_{\hat I_i \times \partial \Omega}} \phi\le C {\inf_{\hat I_i \times \partial \Omega}}
  \phi
  \;
$$
Hence the same is true on each $\mcU_\ell$ and $\hat\mcU_\ell$. Let
$\psi = e^{-s/x}$: it is shown at the beginning of the appendix that one also has
$$
  {\sup_{I_i \times \partial \Omega}} \psi\le C {\inf_{I_i \times \partial \Omega}}
  \psi\;, \qquad
  {\sup_{\hat I_i \times \partial \Omega}} \psi\le C {\inf_{\hat I_i \times \partial \Omega}}
  \psi
$$
(with perhaps a different constant $C$). Once again such
$\ell$-independent inequalities hold on the $\mcU_\ell$'s and
$\hat\mcU_\ell$'s.  At this step it is important to realize that
$$\mathcal L_{\phi,\psi}=L_\phi(p,\partial)=L(p,\phi\partial),$$
where $L$ is {\em uniformly  elliptic} of order 2m  on the relevant cubes.  A scaling and
the usual elliptic interior estimates~\cite[p.~246]{Morrey}
 for the operator $L$ give
$$
 \sum_{i \le k+2m} \int_{\mcU_\ell}\psi^2 \phi^{2i} |\nabla^{(i)}u|_g^2 \le C
\Bigg( \sum_{i \le k} \int_{\hat\mcU_\ell} \psi^2
\phi^{2i}|\nabla^{(i)}L_\phi u|_g^2 +
  \int_{\hat\mcU_\ell} \psi^2| u|^2 \Bigg)\;,
$$
 where $C$ does not depend upon $u$
Summing over $\ell$, we obtain:
$$
\begin{array}{l}

 \sum_{i \le k+2m} \int_{[0,x_0+x_0^2]\times \partial \Omega}\psi^2 \phi^{2i} |\nabla^{(i)}u|_g^2 \\
 \le C
\Bigg( \sum_{i \le k} \int_{[0,x_0+\frac{11}{10}x_0^2]\times \partial \Omega} \psi^2
\phi^{2i}|\nabla^{(i)}L_\phi u|_g^2+
  \int_{[0,x_0+\frac{11}{10}x_0^2]\times \partial \Omega} \psi^2| u|^2 \Bigg)\;.
\end{array}
$$
On the set $G_{-1}$, the norms of $H^l_{\phi,\psi}(\Omega)$ and the usual $H^l$ are equivalent  and interiors estimates also hold.
Thus the preceding inequality is also valid replacing $[0,x_0+x_0^2]\times \partial \Omega$ by $F_{-1}$ and
$[0,x_0+\frac{11}{10}x_0^2]\times \partial \Omega$ by $G_{-1}$. Finaly we use the comparability \eq{firint2} to conclude.
\qed

In the same way, using interior interpolations inequalities and scaling, we prove
\begin{lemma}[interpolation]
Let $0\leq l< k$. For any $\epsilon>0$ there exists $C>0$ such that for any $u\in H^k_{\phi,\psi}$,
$$
\|u\|_{H^l_{\phi,\psi}}\leq \epsilon \|u\|_{H^k_{\phi,\psi}}+C\|u\|_{H^0_{\phi,\psi}}.
$$
\end{lemma}

\noindent{\sc Acknowledgements} :
I am grateful to P. T. Chru\'sciel, J. Corvino, F. Gautero and R. Mazzeo for theirs  useful comments.

\bibliographystyle{amsplain}

\bibliography{../references/newbiblio,%
../references/reffile,%
../references/bibl,%
../references/hip_bib,%
../references/newbib,%
../references/PDE,%
../references/netbiblio,%
../references/erwbiblio,%
 stationary}

\end{document}